\documentclass[a4paper]{article}

 \usepackage{hyperref}
\hypersetup{
    colorlinks=true,
    linkcolor=blue,
    filecolor=blue,
    urlcolor=blue,
    citecolor=blue,
    filecolor=blue
    }

\usepackage{graphicx}
\usepackage{amsmath}
\usepackage{amsthm}
\usepackage{amssymb}
\usepackage{xypic}

\usepackage{xcolor}

\usepackage{enumerate}

\newtheorem{theorem}{Theorem}
\newtheorem{lemma}[theorem]{Lemma}
\newtheorem{proposition}[theorem]{Proposition}%
\newtheorem{corollary}[theorem]{Corollary}%

\newtheorem{example}[theorem]{Example}%
\newtheorem{problem}[theorem]{Problem}%

\theoremstyle{definition}
\newtheorem{definition}[theorem]{Definition}%
\newtheorem{remark}[theorem]{Remark}

\newcommand\fkCan[1]{F_L {\langle #1\rangle} }

\newcommand{\ff}[1]{\widehat{#1}}

\newcommand{\PDL}{\ensuremath{\mathsf{PDL}}\xspace}

\def\Fm{\operatorname{Fm}}
\def\PV{\operatorname{PV}}
\def\At{\operatorname{A}}
\def\Al{{\At}}
\def\AlB{{\operatorname{B}}}
\def\Prog{\operatorname{Prog}}

\def\PDL{{\bf PDL}}
\def\CPDL{{\bf CPDL}}
\newcommand\temp[2]{{#1}^{\mathrm{C}}_{#2}}
\newcommand\trans[2]{{#1}^{+}_{#2}}

\def\val{\vartheta}
\newcommand{\Sub}{\mathop{\mathsf{Sub}}}

\newcommand\comm[1]\empty

\def\emp{\varnothing}

\newcommand\DiA[1]{\langle #1\rangle}
\def\vf{\varphi}

\newcommand\clP[1]{\mathcal{P}(#1)}

\def\AA{\forall}

\newcommand{\clF}{\mathcal{F}}

\def\clM{\mathcal{M}}

\newcommand{\Log}{\mathop{Log}}

\def\mo{\vDash}
\def\imp{\rightarrow}

\def\wK4{\logicts{wK4}}
\def\vL{L}

\def\Alg{\mathop{Alg}}

\def\WS5{\logicts{WS5}}

\newcommand\logicts[1]{{\textsc{#1}}}
\newcommand{\LK}[1]{\logicts{K#1}}
\newcommand{\LS}[1]{\logicts{S#1}}

\def\vf{\varphi}
\def\emp{\varnothing}

\newcommand\hide[1]{\empty}

\def\ToDo{\todo}

\newcommand\ISH[1]{{~\bf IS}: {\color{teal} #1}}
\renewcommand\ISH[1]\empty
\newcommand\DR[1]{{~\bf DR}: {\color{blue}#1}}
\renewcommand\DR[1]\empty

\newcommand\obsolete[1]{{\color{darkgray}\noindent{\bf Obsolete:} #1}}
\renewcommand\obsolete[1]\empty

\newcommand\later[1]{{\color{lightgray}\noindent{\bf later:} {#1} }}
\renewcommand\later[1]\empty

\author{
  Daniel Rogozin \\
  University College London \\
  \texttt{d.rogozin@ucl.ac.uk}
  \and
  Ilya Shapirovsky \\
  New Mexico State University \\
  \texttt{ilshapir@nmsu.edu}
}
\date{}
\title{On decidable extensions of Propositional Dynamic Logic with Converse}

\begin{document}

\maketitle



\begin{abstract}
We describe a family of decidable propositional dynamic logics,
where atomic modalities satisfy some extra conditions
(for example, given by axioms of the logics $\LK{5}$, $\LS{5}$, or $\LK{45}$
for different atomic modalities).
It follows from recent results \cite{KSZ:AiML:2014}, \cite{KikotShapZolAiml2020} that
if a modal logic $L$ admits a special type of filtration (so-called definable filtration),
then its enrichments with modalities for the transitive closure and converse relations also admit
definable filtration. We
use these results to show that if logics $L_1, \ldots , L_n$ admit definable filtration,
then the propositional dynamic logic with converse extended by the fusion $L_1*\ldots * L_n$
has the finite model property.

\hide{
consider the case when $L$ is the fusion $L_1*\ldots * L_n$ of logics, and show that
if all $L_i$ admit definable filtration, then the Propositional Dynamic
Logic (with converses) extended with the axioms of $L_1,\ldots, L_n $ has the finite model property. In particular,
it follows that if $L_i$ are finitely axiomatizable, then the corresponding extension of Propositional Dynamic
Logic is decidable.

\todo{..}
We apply this result to the case when
$L$ is the fusion $L_1*\ldots L_n$ of different logics (e.g., of ) to show that
the extension of Propositional Dynamic
Logic (with converses) with the axioms of $L$ is complete with respect to its standard finite models.

\todo{..}
Using recent
transfer results for
logics enriched with modalities for the transitive closure and converse relations \cite{KSZ:AiML:2014}, \cite{KikotShapZolAiml2020},
we show that many extensions of Propositional Dynamic
Logic (with converses) is complete with respect to its standard finite models.

\todo{..}We construct a family of decidable propositional dynamic logics, where atomic modalities satisfy some extra conditions (for example, given by axioms of the logics K5, S5, or K45 for different atomic modalities). Using recent transfer results about the finite model property for multimodal logics that admit filtration [Kikot, Shapirovsky, Zolin, AiML2020], we show that many extensions of Propositional Dynamic Logic (with converses) are complete with respect to its standard finite models.
}

\end{abstract}


\section{Introduction}
The Propositional Dynamic Logic with Converse is known
to be complete with respect to its standard finite models, and hence is decidable \cite{Parikh1978CPDL}.
We generalize this result for a family of normal extensions of this logic.

Let $\CPDL(\Al)$ be the
propositional dynamic logic with converse modalities, where $\Al$ indicates the set of atomic modalities.
For a set of modal formulas $\Psi$ in the language of $\Al$,
let $\CPDL(\Al)+\Psi$ be the normal extension of $\CPDL(\Al)$ with $\Psi$.

In \cite{KSZ:AiML:2014} and \cite{KikotShapZolAiml2020},
it was shown that
if a modal logic $L$ admits a special type of filtration (so-called definable filtration),
then its enrichments with modalities for the transitive closure and converse relations also admit
definable filtration.  In particular, it follows that if a logic $L$ admits definable filtration,
then $\CPDL(\Al)+L$ has the finite model property.

We will be interested in the case when $\CPDL(\Al)$ is extended by a fusion of logics $L=L_1*\ldots * L_n$.
For example, $\CPDL(\Diamond_1,\Diamond_2,\Diamond_3)+\LK{5}*\LK{45}*\LK{4}$ is the extension
of $\CPDL(\Al)$, where the first and the second atomic modalities satisfy the principle $\Diamond p\imp \Box \Diamond p$,
the second and the third satisfy $\Diamond\Diamond p\imp \Diamond p$.
We show in Theorem \ref{thm:mainTransferNew} that if the logics $L_i$ admit definable filtration, then
their fusion admits definable filtration as well. It follows that in this case
$\CPDL(\Al)+L_1*\ldots * L_n$ has the finite model property, and, if all $L_i$ are finitely axiomatizable,
$\CPDL(\Al)+L_1*\ldots * L_n$ is decidable (Corollary \ref{cor:main}).
Consequently, we have the following decidability result (Corollary \ref{cor:final}):
if each $L_i$
is
\begin{itemize}
\item
one of the logics
$$\LK{},~\logicts{T},~\LK{4},~\LS{4},~
\LK{} + \{\Diamond^m p \to \Diamond p\}~(m \geq 1),$$ or an extension of any of these logics with a variable-free formula,
\hide{
\item
one of the logics
$\LK{},~\logicts{T},~\LS{4},~\LK{}+\{p\imp \Box\Diamond p\},~\LK{}+\{\Diamond\top\},~
\LK{4}+\{\Diamond\top\}$,
$\LK{} + \{\Diamond^m p \to \Diamond p\}$ for $m \geq 1$,  or
}
\item  locally tabular (e.g., $\LK{5},~\LK{45},~\LS{5}$, the difference logic), or
\item  a {\em stable logic} (defined in \cite{bezhanishvili2016stable}), or
\item axiomatizable by canonical MFP-modal formulas (defined in \cite{KikotShapZolAiml2020}),
\end{itemize}
then $\CPDL(\At)+L_1 * \ldots * L_n$ has the finite model property;
if also all $L_i$ are finitely axiomatizable, then $\CPDL(\At)+L_1 * \ldots * L_n$ is decidable.
Some particular instances of this fact (in the language without converse modalities) were known before:
for the case when each $L_i$ is a stable logic, it was announced in \cite{IlinAiML2016};
the case when each $L_i$ is axiomatizable by canonical MFP-modal formulas follows from \cite[Corollary 4.13]{KikotShapZolAiml2020}.

The paper is organized as follows. Section \ref{sec:prel} provides basic syntactic and semantic definitions.
Section \ref{sec:transfer} is an exposition of necessary transfer results from \cite{KSZ:AiML:2014} and \cite{KikotShapZolAiml2020}.
Main results (Theorem \ref{thm:mainTransferNew}, Corollary \ref{cor:main}, and Corollary \ref{cor:final}) are given in Section \ref{sec:main}.

\medskip
 A preliminary report on some results of this paper was given in \cite{RogozinShapAiML2022}.

\section{Syntactic and semantic preliminaries}\label{sec:prel}

We assume that the reader is familiar with basic notions 
of modal logic \cite{B:R:V:ML:2002, Ch:Za:ML:1997,Goldblatt1992LogOfTime}. Below we briefly recall some of them and fix notation.

\paragraph{Normal logics and Kripke semantics.}
Fix a set
$\PV = \{ p_i \: \mid \: i < \omega \}$  of {\em propositional variables}. For a set $\Al$, 
the {\em set of modal $\At$-formulas} $\Fm(\At)$
is built from propositional variables using Boolean connectives
$\bot,\to$ and unary connectives $\DiA{a}$ for $a\in \Al$ ({\em modalities}). Other connectives are defined in the standard way, in particular $[a]$ abbreviates  $\neg\DiA{a} \neg$.
Sometimes we write $\Diamond_a$ for $\DiA{a}$ and
$\Box_a$ for $[a]$.
If $\Al$ is a singleton $\{a\}$, we write $\Diamond$ and $\Box$ for $\DiA{a}$ and $[a]$, respectively.
\hide{
 The
{\em set of modal $\At$-formulas} $\Fm(\At)$ is generated by the following grammar:

\begin{center}
    $\varphi ::= \bot \: \mid \: p    \: \mid \: (\varphi \to \psi) \: \mid \: \DiA{a}  \varphi$ \quad $(
    p\in \PV, \; a \in \At)$
\end{center}
Other connectives are defined in the standard way,
in particular $[a]$ abbreviates  $\neg\DiA{a} \neg$.
}


A {\em (normal) modal $\At$-logic} is a set of formulas $L\subseteq\Fm(\At)$ such that:
\begin{enumerate}
\item $L$ contains all Boolean tautologies;
\item For all $a \in \Al$, $\DiA{a} \bot \leftrightarrow \bot \in L$ and $\DiA{a}  (p \lor q) \leftrightarrow \DiA{a}  p \lor \DiA{a}  q \in L$;
\item $L$ is closed under the rules of Modus Ponens,  uniform substitution,
and {\em monotonicity}:
$\vf \to \psi\in L$ implies
$\DiA{a} \vf \to \DiA{a} \psi\in L$
for all $a\in \At$.
\end{enumerate}
For an $\At$-logic $L$ and a set $\Psi$ of $\At$-formulas,
$L+\Psi$ is the smallest  modal $\At$-logic that contains $L\cup \Psi$. As usual, the smallest  unimodal logic is denoted by $\LK{}$.

An $\Al$-frame is a structure $F = (W, (R_a)_{a \in \Al})$,
where each $R_a$ is a binary relation on $W$.
A model on an $\Al$-frame
is a structure $M = (F, \vartheta)$,
where $\vartheta : \PV \to \clP{W}$, where $\clP{W}$
is the set of all subsets of $W$. The truth definition is standard:
\begin{itemize}
    \item $M, x \models p_i$ iff $x \in \vartheta(p_i)$;
    \item $M, x \not\models \bot$;
    \item $M, x \models \varphi \to \psi$ iff $M, x \not\models \varphi$ or $M, x \models \psi$;
    \item $M, x \models \DiA{a}  \varphi$ iff there exists $y$ such that $xR_a y$ and  $M, y \models \varphi$.
\end{itemize}
We set $M \models \varphi$ iff $M, x \models \varphi$ for all $x$  in $M$, and
$F\mo \vf$ iff $M\mo\vf$ for all $M$ based on $F$; $\Log(F)$ is the set
$\{\vf\in \Fm(\At) \mid F\mo \vf\}$. For a class $\clF$ of frames, $\Log(\clF)=\bigcap\{\Log(F)\mid F\in\clF\}$.
A logic $L$ is {\em Kripke complete} iff it is
characterized  by a class $\clF$ of frames, that is
$L = \Log(\clF)$. A logic $L$ has the {\em finite model property} iff it is characterized by a class of finite models, or equivalently, by
a class of finite frames (see, e.g., \cite[Theorem 3.28]{B:R:V:ML:2002}).
\later{Remove unused notation}

For a logic $L$, $\operatorname{Mod}(L)$ is
the class of models such that $M \models L$, i.e., $M \models \varphi$ for all $\varphi \in L$.

\later{
If $L$ is a set of formulas, then $M \models L$ stands for $M \models \varphi$ for all $\varphi \in L$.
}
\later{
The class of $L$-frames, $\operatorname{Frames}(L)$ consists of frames that validate $L$. The logic of a class of frames $\mathcal{F}$, $\operatorname{Log}(\mathcal{F})$,
is the set of formulas valid in each of those frames.
}

\paragraph{Propositional Dynamic Logics.}
Let $\At$ be finite.
The set $\Prog(\At)$ ({\em ``programs''}) is generated by the following grammar:
\begin{center}
    $e ::= a \: \mid \:  (e \cup e)  \: \mid \:   (e \circ e)  \: \mid \:   e^+ $ \quad \text{ for }$a \in \At$
\end{center}
\begin{remark} Our language of programs is test-free.
\end{remark}
\later{More details on (in)finiteness of the alphabet}

\begin{definition}
A {\em normal propositional dynamic $\At$-logic}
is a normal $\Prog(\At)$-logic
that contains the following formulas for all $  e ,  f   \in \Prog(\At)$:
\begin{enumerate}
\item[\bf A1] $\langle e \cup f \rangle p \leftrightarrow \langle e \rangle p \vee \langle f \rangle p$,
\item[\bf A2] $\langle e \circ f \rangle p \leftrightarrow \langle e \rangle \langle f \rangle p$,
\item[\bf A3] $\langle e \rangle p \to \langle e^{+} \rangle p$,
\item[\bf A4] $\langle e \rangle \langle e^{+} \rangle p \to \langle e^{+} \rangle p$,
\item[\bf A5] $\langle e^{+} \rangle p \to \langle e \rangle p \vee \langle e^{+} \rangle (\neg p \land \langle e \rangle p)$.
\end{enumerate}
The least normal propositional dynamic $\At$-logic
is denoted by $\PDL(\At)$.

\smallskip
We also consider dynamic logics  with converse modalities.
The set $\Prog_t(\At)$ is given by the following grammar:
\begin{center}
    $  e  ::= a \: \mid \:   (e \cup e)  \: \mid \:  (e \circ e) \: \mid \:  e^+  \:\mid  e^{-1}  \quad $ \text{ for } $a\in \At$
\end{center}
\later{\todo{Is everything consistent with angles and boxes?}}
A {\em normal propositional dynamic  $\At$-logic with converse modalities}
is a normal $\Prog_t(\At)$-logic
that
contains the formulas {\bf A1}--{\bf A5} and the formulas
\begin{enumerate}
\item[\bf A6] $p \to [e]\langle e^{-1} \rangle p$
\item[\bf A7] $p \to [e^{-1}]\langle e \rangle p$
\end{enumerate}
for all $  e ,  f   \in \Prog_t(\At)$.
The smallest dynamic $\At$-logic with converses is denoted by $\CPDL(\At)$.
\end{definition}

The validity of formulas
{\bf A1}-{\bf A7} in a frame {$(W,(R_e)_{e\in \Prog_t(\At)})$}
is equivalent to the following identities:
\begin{eqnarray}
\label{eq:SegConditions}
&&R_{(e\circ f)}=R_{ e }\circ R_{  f},~
R_{(e\cup f)}=R_{ e }\cup R_{ f},~
R_{e^+}=(R_{  e})^+,~\\
\label{eq:TempConditions}
&&R_{ e^{-1}  }=(R_{  e })^{-1},
\end{eqnarray}
\later{More details, since we have a mixture: an $L$-model and frame conditions}
where $R^+$ denotes the transitive closure of $R$, $R^{-1}$ the converse of $R$;
models based of such frames are called {\em standard}; see, e.g., \cite[Chapter 10]{Goldblatt1992LogOfTime}.
It is known that $\CPDL(\At)$ is
complete with respect to its standard finite models \cite{Parikh1978CPDL}.
Our aim is to prove this for a family of extensions of $\CPDL(\At)$.

\section{Filtrations and decidable extensions of dynamic logic}\label{sec:transfer}

\subsection{Logics that admit definable filtration}
For a model $M = (W, (R_a)_{a \in \Al}, \vartheta)$ and a set of formulas $\Gamma$, put
\begin{center}
$x\sim_{\Gamma} y$ \quad iff \quad$\forall \psi\in\Gamma\; (M,x\models \psi \Leftrightarrow M,y\models \psi)$.
\end{center}
The equivalence $\sim_\Gamma$ is said to be {\em induced by $\Gamma$ in $M$}.

For $\varphi \in \operatorname{Fm}(\Al)$, let $\operatorname{Sub}(\varphi)$ be the set of all subformulas of $\varphi$. A set $\Gamma$  of formulas is {\em $\operatorname{Sub}$-closed}, if $\varphi \in \Gamma$ implies $\operatorname{Sub}(\varphi) \subseteq \Gamma$.
\begin{definition}
Let $\Gamma$ be a $\operatorname{Sub}$-closed
set of formulas.
A {\em $\Gamma$-filtration} of a model $M = (W, (R_a)_{a \in \Al}, \vartheta)$   is a model $\ff{M}=(\ff{W},(\ff{R}_a)_{a \in \Al},\ff{\theta})$
s.t.
\begin{enumerate}
\item $\ff{W}=W/{\sim}$ for some equivalence relation $\sim$ such that $\sim \;\subseteq \;\sim_\Gamma$, i.e.,
\begin{center}
$x\sim  y$ \quad implies \quad
 $\forall \psi\in\Gamma\; (M,x\models \psi \Leftrightarrow M,y\models \psi)$.
\end{center}
\item ${\ff{M},[x]\models p}$ iff ${M,x\models p}$, for all $p\in \Gamma$.
Here $[x]$ is the class of $x$ modulo $\sim$.
\item For all $a \in \Al$, we have ${(R_a)}_{\sim} \subseteq \ff{R}_a \subseteq  {(R_a)}_{\sim}^\Gamma$, where
$$
\begin{array}{ccl}
~[x]\,{(R_a)}_\sim\,[y] & \text{iff} & \exists x'\sim x\ \exists y'\sim
y\;
(x'\,R_a\,y'),
\smallskip \\
~[x]\,{(R_a)}_{\sim}^\Gamma\,[y] & \text{iff} & \forall \psi\;
   (\DiA{a} \psi\in \Gamma \: \& \: M,y\models\psi \Rightarrow M,x\models \DiA{a}\psi ).
\end{array}
$$
\end{enumerate}
The relations ${(R_a)}_\sim$ and ${(R_a)}_\sim^\Gamma$ on $\widehat{W}$ are called
the \emph{minimal} and the \emph{maximal filtered relations}, respectively.
\end{definition}

If $\sim\; =\; \sim_\Delta$ for some finite 
set of formulas ${\Delta\supseteq\Gamma}$,
then $\ff{M}$ is called a \emph{definable $\Gamma$-filtration} of the model~$M$.
If $\sim\; =\; \sim_\Gamma$, the filtration $\ff{M}$ is said to be {\em strict}.

The following fact is standard:
\begin{lemma}[Filtration lemma]\label{Lemma:Filtration}
Suppose that $\Gamma$ is a finite $\operatorname{Sub}$-closed set of formulas
and $\ff{M}$ is a $\Gamma$-filtration of a model~$M$.
Then, for all points ${x\in W}$ and all formulas ${\varphi \in\Gamma}$,
we have: \ \
\begin{center}
${M,x\models \varphi}$ iff ${\widehat{M},[x]\models\varphi}$.
\end{center}
\end{lemma}
\begin{proof}
Straightforward induction on $\vf$.
\end{proof}

\begin{definition}
We say that a class $\mathcal{M}$  of Kripke models \emph{admits  definable (strict) filtration}
iff for any ${M \in \mathcal{M}}$
and for any finite $\operatorname{Sub}$-closed set of formulas~$\Gamma$,
there exists a finite model in $\mathcal{M}$ that is a definable (strict) $\Gamma$-filtration of~$M$.
A logic {\em admits definable (strict) filtration} iff the class
$\operatorname{Mod}(L)$
of its models does.
\end{definition}

It is immediate from the Filtration lemma
that if a logic admits filtration, then  it has the finite model property.

Strict filtrations are the most widespread in the literature;
for example, it is well-known that the logics
$\LK{},~\logicts{T},~\LK{4},~\LS{4},~\LS{5}$ admit strict filtration,
 see e.g., \cite{Ch:Za:ML:1997}.  Constructions where the initial equivalence is refined
 were also used since the late 1960s \cite{Segerberg1968}, \cite{Gabbay:1972:JPL}, and
 later, see, e.g., \cite{Shehtman:AiML:2004}.\later{improve}
 Refining the initial equivalence makes the filtration method much more flexible.
 For example, it is not difficult to see that the logic $\LK{5}={\LK{}+\{\Diamond p\imp \Box \Diamond p\}}$ does not admit strict filtration.
 However, $\LK{5}$ admits definable filtration, see, e.g., \cite[Theorem 5.35]{Ch:Za:ML:1997}. Another explanation is that
 $\LK{5}$ is locally tabular \cite{nagle_thomason_1985},
 and every locally tabular logic admits definable filtration, see Section \ref{sec:LTimpADF} for details.

Notice that if a logic $L$ admits definable filtration, then its extension with a variable-free formula $\vf$ admits definable filtration as well (for a given $L+\{\vf\}$-model $M$ and $\Gamma$,
consider a $\Gamma\cup\operatorname{Sub}(\varphi)$-filtration).

\later{\ToDo{K5 is LT: Double Check the ref}; seems to be ok, see Corollary 5}

\subsection{Transferring admissibility of definable filtration}
In \cite{KSZ:AiML:2014} and \cite{KikotShapZolAiml2020}, definable filtrations were used to obtain
transfer results for logics enriched with modalities for the transitive closure
and converse relations.

Let $e\in \Al$.
For an $\Al$-logic $L$, let $\trans{L}{e}$ be the extension
of the logic
$L$ with axioms {\bf A3}, {\bf A4}, and {\bf A5}, and
let $\temp{L}{e}$ be the extension of $L$ with the axioms
{\bf A6} and {\bf A7}.

For an $\Al$-model $M=(W,(R_a)_{a\in\Al},\val)$,
let $\temp{M}{  e }$ be its expansion with the converse of $R_e$:
$$\temp{M}{  e }=(W,(R_a)_{a\in\Al},R_e^{-1}, \val);$$
similarly,  $\trans{M}{  e }$ denotes the expansion of $M$ with the transitive closure of $R_{ e }$:
$$\trans{M}{  e }=(W,(R_a)_{a\in\Al},R_{  e } ^+,\val).$$
It is straightforward from \eqref{eq:SegConditions}
and
\eqref{eq:TempConditions}
that if $M$ is an $L$-model, then
$\trans{M}{  e }$ is a model of $\trans{L}{ e }$,
and $\temp{M}{ e }$ is a model of $\temp{L}{e }$.

Assume that a logic $L$ admits definable filtration.
In \cite[Theorem 3.9]{KikotShapZolAiml2020}, it was shown that in this case
the logic $\trans{L}{ e }$ admits definable filtration as well.
This crucial result implied that
$\PDL(\At)+L$ has the finite model property, and if also $L$ is finitely axiomatizable, then
$\PDL(\At)+L$ is decidable \cite[Theorem 4.6]{KikotShapZolAiml2020}.

If follows from
\cite[Theorem 2.4]{KSZ:AiML:2014} that if $\vL$ admits definable filtration,
then so does $\temp{L}{ e}$.
\begin{remark}
Theorem \cite[Theorem 2.4]{KSZ:AiML:2014} was formulated for frames, not for models; however, the definable filtrations given
in the proof of this theorem work for models without any modification.
\end{remark}

\obsolete{
The following theorem is a corollary of
\cite[Theorem 4.6]{KikotShapZolAiml2020} and the proof of \cite[Theorem 2.4]{KSZ:AiML:2014}.
\ToDo{details}
\ISH{This requires some more work, since
 \cite[Theorem 2.4]{KSZ:AiML:2014} was given for different typo of filtrations. It is on me.
}
}
\begin{theorem}[\cite{KikotShapZolAiml2020},\cite{KSZ:AiML:2014}] \label{thm:ADFforCPDL}
Let  $\AlB$  be a subset of a finite\later{``finite'' should be recursive for A, and finite for B}
set $\Al$.
If a $\AlB$-logic $L$ admits definable filtration, then
$\CPDL(\At)+L$ has the finite model property.
If also $L$ is finitely axiomatizable, then $\CPDL(\At)+L$ is decidable.
\end{theorem}
\later{``$A$ is finite'': does it matter?}

\section{Filtrations for fusions}\label{sec:main}

\subsection{Fusions}
Let $L_1, \ldots, L_n$ be logics in languages
that have mutually disjoint sets of modalities. The {\em fusion} $L_1 * \ldots * L_n$ is the smallest logic that contains $L_1, \ldots, L_n$.
We adopt the following convention: for logics $L_1,\ldots, L_n$ in the same language,
we also write $L_1 * \ldots * L_n$  assuming that we ``shift'' modalities; e.g.,
 $\LK{5}*\LK{5}$ denotes the bimodal logic given by the two axioms
 $\Diamond_i p\imp \Box_i\Diamond_i p$, $i=1,2$. \later{$i=0,1$}

It is known that the fusion of consistent modal logics is a conservative extension of its components \cite{thomason1980independent}.
Also, the fusion operation preserves
Kripke completeness, decidability, and the finite model property \cite{kracht1991properties,fine1996transfer,wolter1996fusions}.
\later{\ToDo{Double check the second ref}}

In \cite{KikotShapZolAiml2020}, it was noted that if canonical logics $L_1,\ldots, L_n$ admit strict filtration,
then the fusion $L=L_1*\ldots *L_n$ admits strict filtration; it follows from
Theorem \ref{thm:ADFforCPDL} that $\CPDL(\At)+L$ has the finite model property for the case of such $L$.

\begin{example}
The logic $\CPDL(\Diamond_1,\Diamond_2)+\LS{4}*\LS{5}$ has the finite model property and decidable.
\end{example}
\hide{
\todo{redo}
Moreover, this observation can be generalized as follows:
if $L$ admits strict filtration and $L'$ admits definable filtration, than
the fusion $L*L'$ admits definable filtration.

\begin{example}
The logic $\CPDL(\Di_1,\Di_2,\Di_3)+\LS{4}*\LS{5}*\LK{5}$ has the finite model property and decidable.
\end{example}
}
\noindent
It does not cover many important examples where logics $L_i$ do not admit strict filtration
(like in the case of the logic $\LK{5}*\LK{5}$).
We will show  below that the admissibility of definable filtration is preserved under the operation of fusion, that
extends applications of Theorem \ref{thm:ADFforCPDL} significantly.

\subsection{Main result}
Recall that a set of formulas $\Psi$ is {\em valid} in a modal algebra $B$, in symbols $B\mo \Psi$, iff
$\vf=1$ holds in $B$ for every $\vf\in \Psi$.

For a model $M=(W,(R_a)_{a\in \Al},\val)$ and an $\Al$-formula $\vf$, put $\vf_M=\{x\mid M,x\mo\vf\}.$
Let $D(M)=\{\vf_M\mid \vf\in \Fm(\At)\}$
 be the set of definable subsets of $M$, considered as a Boolean subalgebra of the powerset algebra $\clP{W}$, and
let
$\Alg(M)$ be the modal algebra  $(D(M),(f_a)_{a\in \Al})$, where
$f_a(V)=R_a^{-1}[V]$ for $V\subseteq W$.
The following fact is standard: if $L$ is a logic, then
\begin{equation}\label{eq:model-alg}
M\mo L \text{ iff } \Alg(M)\mo L
\end{equation}
(``if'' is trivial, ``only if'' follows from the fact that logics are closed under substitutions).
If $M'=(W,(R_a)_{a\in \Al},\val')$ is a model such that
$\val'(p)\in D(M)$ for all variables $p$, then
it follows from \eqref{eq:model-alg} that
\begin{equation}\label{eq:robustmodel}
\text{if $M\mo L$, then $M'\mo L$};
\end{equation}
indeed, $\Alg(M')$ is a subalgebra of $\Alg(M)$.

\begin{proposition}\label{prop:extensible-filtr}
Let $\Gamma$ be a $\operatorname{Sub}$-closed
set of formulas,
$M = (W, (R_a)_{a \in \Al}, \vartheta)$ a model. If
$\ff{M}=(W/{\approx},(\ff{R}_a)_{a \in \Al},\ff{\theta})$ is
a $\Gamma$-filtration of $M$ for some equivalence $\approx$, then for every equivalence $\sim$ finer than $\approx$
there exists a $\Gamma$-filtration $\ff{M}'$ of $M$ such that $W/{\sim}$ is the carrier of $\ff{M}'$ and
\begin{equation}\label{eq:extensible-filtr}
\ff{M}\mo\vf \text{ iff } \ff{M}'\mo \vf
\end{equation}
for every $\vf\in\Fm(\Al)$.
\end{proposition}
\newcommand\crs[1]{{#1}^{\approx}}
\begin{proof}
Since $\sim \;\subseteq \;\approx$, for every $u \in W/{\sim}$ there exists a unique
element of $\widehat{M}$ that contains $u$; we denote it by
$\crs{u}$. The binary relations $\ff{R}_a'$ in $\ff{M}'$
and the valuation
$\ff{\theta}'$
are defined as follows:
\begin{align*}
    &\ff{R}_a'=\{(u,v) \mid  (\crs{u},\crs{v})\in  \ff{R}_a \};\\
    &\ff{\theta}'(p)=\{u\in W/{\sim}\mid \crs{u}\in \ff{\theta}(p)\} \text{ for } p\in \PV.
\end{align*}
It is straightforward that the map $u\mapsto u\Delta$ is a p-morphism of a model $\ff{M}'$ onto  $\ff{M}$.
By the p-morphism lemma (see,
e.g., \cite[Section 1]{Goldblatt1992LogOfTime}), we have
\begin{equation}\label{eq:uDelta}
\ff{M}',u\mo\vf \text{ iff } \ff{M},\crs{u}  \mo \vf.
\end{equation}
Now \eqref{eq:extensible-filtr} follows.

Trivially, $\sim\;\subseteq \;\sim_\Gamma$. The second filtration condition
follows from the definition of $\ff{\theta}'$.  Let $a\in\Al$.
For $x\in W$, let $[x]_\approx$ and  $[x]_\sim$  be the classes of $x$ modulo $\approx$ and $\sim$, respectively.
If $xR_a y$, then $[x]_\approx \ff{R}_a [y]_\approx$, because $\ff{M}$ is a filtration of $M$;
now $[x]_\sim \ff{R}_a' [y]_\sim$  by the definition of $\ff{R}_a'$.
That $\ff{R}'_a$ is contained in the maximal filtered relation follows from \eqref{eq:uDelta}.
\end{proof}
\hide{
\begin{remark}
This proposition is a version of a construction called {\em refinement}, see, e.g., \cite[Chapter 8]{Ch:Za:ML:1997}.
In the above proof,
the map $u\mapsto u\Delta$ is a {\em p-morphism} of models, and
the algebras $\Alg{\ff{M}}$ and $\Alg{\ff{M}'}$ are isomorphic.
\end{remark}
}

The following is a generalization of \cite[Theorem 4.8]{KikotShapZolAiml2020}.
\begin{theorem}\label{thm:mainTransferNew}~ 
  If logics $L_1$ and $L_2$ admit definable filtration, so does $L_1*L_2$.
\end{theorem}
\begin{proof}
Let $\Al$ and $\AlB$ be alphabets of modalities of the logics $L_1$ and $L_2$, respectively.
Without loss of generality we may assume that
  $\Al$ and $\AlB$ are disjoint.

  Consider an $L_1*L_2$-model $M  = (W, (R_a)_{a\in \Al}, (R_b)_{b\in \AlB}, \vartheta)$,
    and a finite $\Sub$-closed set of formulas $\Gamma\subset \Fm(\Al\cup\AlB)$.
    Consider a set of fresh variables $V = \{ q_{\varphi} \mid \varphi \in \Gamma \}$, and
    define a valuation $\eta$ in $W$ as follows:
for $q_\vf\in V$, let $\eta(q_\vf)=\{x\mid M,x\mo\vf\}$; otherwise, put  $\eta(q)=\emp$.
Let $M_V=(W, (R_a)_{a\in \Al}, (R_b)_{b\in \AlB}, \eta)$.
We have:
\begin{equation}\label{eq:red0}
  D(M_V)\subseteq D(M),
\end{equation}
and  by \eqref{eq:robustmodel},
\begin{equation}\label{eq:red1}
M_V\mo L_1*L_2.
\end{equation}
Consider the $\Al$- and $\AlB$-reducts of  $M_V$:
$$M_\Al=(W, (R_a)_{a\in \Al}, \eta),\quad M_\AlB=(W, (R_b)_{b\in \AlB}, \eta).$$
It follows from \eqref{eq:red1} that
\begin{equation}\label{eq:red2}
M_\Al\mo L_1, \quad M_\AlB\mo L_2.
\end{equation}
Consider the following sets of formulas:
$$\Gamma_\Al = V \cup \{ \DiA{a} q_{\varphi} \mid \DiA{a} \varphi \in\Gamma\,\& \,a   \in \Al\}, \quad
\Gamma_\AlB = V \cup \{ \DiA{b} q_{\varphi} \mid \DiA{b} \varphi \in\Gamma\,\& \,b   \in \AlB \}.$$
Since logics $L_1$ and $L_2$ admit definable filtration, there are finite sets $\Delta_\Al$ and
$\Delta_\AlB$ of formulas, and models $\ff{M}_\Al$, $\ff{M}_\AlB$ such that
\begin{align}
&\ff{M}_\Al\mo L_1,& &\ff{M}_\AlB\mo L_2,  \label{eq:red3} \\
&\Gamma_\Al\subseteq \Delta_\Al\subset \Fm(\Al),&
&\Gamma_\AlB\subseteq \Delta_\AlB\subset \Fm(\AlB),& \\
&\text{$\ff{M}_\Al$ is a $\Gamma_\Al$-filtration of $M_\Al$},&
&\text{$\ff{M}_\AlB$ is a $\Gamma_\AlB$-filtration of $M_\AlB$},& \\
&\text{the carrier of $\ff{M}_\Al$ is $W/{\sim_\Al}$,}&
&\text{the carrier of $\ff{M}_\AlB$ is $W/{\sim_\AlB}$,}&
\end{align}
where $\sim_\Al$ is the equivalence on $W$ induced by $\Delta_\Al$ in $M_\Al$,
and  $\sim_\AlB$ is the equivalence on $W$ induced by $\Delta_\AlB$ in $M_\AlB$.
Let $\sim$ be the equivalence $\sim_\Al\cap \sim_\AlB$.
By Proposition \ref{prop:extensible-filtr} and \eqref{eq:red3}, there are models
$\ff{M}_\Al'$ and $\ff{M}_\AlB'$ whose carrier is $W{/}{\sim}$ such that
\begin{align}
&\ff{M}_\Al'\mo L_1,& &\ff{M}'_\AlB\mo L_2,  \label{eq:red3-1}  \\
&\text{$\ff{M}'_\Al$ is a $\Gamma_\Al$-filtration of $M_\Al$},& \label{eq:red3-2}
&\text{$\ff{M}'_\AlB$ is a $\Gamma_\AlB$-filtration of $M_\AlB$}.&
\end{align}
Notice that $\Gamma_\Al$ and $\Gamma_\AlB$ contain the same variables, namely $V$.
The value of any variable in $V$ is the same in $\ff{M}_\Al'$ as in $\ff{M}_\AlB'$.
Also,  we can assume that
the values of variables not in $V$ are empty in these models: making them empty
does not affect \eqref{eq:red3-1} by \eqref{eq:robustmodel}, and \eqref{eq:red3-2}  by the definition of filtration. Consequently, we can assume that $\ff{M}'_\Al$ and $\ff{M}'_\AlB$ have the same valuation:
\begin{equation}
\ff{M}_\Al' = (W/{\sim},(\ff{R}_a)_{a\in\Al},\ff{\eta}),\quad
\ff{M}_\AlB' = (W/{\sim},(\ff{R}_b)_{b\in\AlB},\ff{\eta}).
\end{equation}
By \eqref{eq:red3-2},  the model
\begin{equation}\label{eq:ref-Mv}
\ff{M}_V = (W/{\sim},(\ff{R}_a)_{a\in\Al}, (\ff{R}_b)_{b\in\AlB},\ff{\eta}) \text{ is a $(\Gamma_\Al\cup \Gamma_\AlB)$-filtration of $M_V$}.
\end{equation}
By \eqref{eq:red3-1},
\begin{equation}\label{eq:red4}
  \ff{M}_V \mo L_1*L_2.
\end{equation}
Finally, let $\ff{M}=(W/{\sim},(\ff{R}_a)_{a\in\Al}, (\ff{R}_b)_{b\in\AlB},\ff{\theta})$,
where $\ff{\theta}(p)=\ff{\eta}(q_p)$  for $p\in \Gamma$, and $\ff{\theta}(p)=\emp$ otherwise.
By \eqref{eq:red4} and \eqref{eq:robustmodel},
$$
\ff{M} \mo L_1*L_2.
$$

Let us show that $\ff{M}$ is a definable $\Gamma$-filtration of $M$.

First, observe that $\sim$ is induced in $M_V$ by the set $\Delta_\Al\cup\Delta_\AlB$, and so it is induced in $M$ by a set of formulas according to (\ref{eq:red0}).
Since $V\subseteq \Delta_\Al\cup\Delta_\AlB$, the equivalence $\sim$ refines
the equivalence $\sim_\Gamma$ induced in $M$ by $\Gamma$.

Let $c\in \Al\cup\AlB$.
That $\ff{R}_c$ contains the corresponding minimal filtered relation follows from   \eqref{eq:ref-Mv}.
Let us  show that $\ff{R}_c$ is contained in the maximal filtered relation $(R_c)_\sim^\Gamma$.
 Notice that by the definition of $\eta$, for every $\vf\in\Gamma$, $z\in W$,
 \begin{equation}\label{eq:red-eta-di}
 M_V,z\mo q_\vf \text{ iff } M,z\mo \vf, \text{ and hence }
 M_V,z\mo \DiA{c} q_\vf \text{ iff } M,z\mo \DiA{c} \vf.
 \end{equation}
Consider $\sim$-classes $[x]$, $[y]$ of $x,y\in W$, and assume that $\DiA{c}\vf\in \Gamma$ and $M,y\mo \vf$.
By \eqref{eq:red-eta-di},  $M_V,y\mo q_\vf$.
We have $\DiA{c} q_\vf\in \Gamma_\Al\cup\Gamma_\AlB$, so by \eqref{eq:ref-Mv}, $M_V,x\mo \DiA{c} q_\vf$. By \eqref{eq:red-eta-di} again, $M,x\mo\DiA{c}\vf$.
\end{proof}

\begin{example}
By the above theorem, $\LK{5}*\LK{5}$ admits definable filtration.
Consequently, the logic $\CPDL(\Diamond_1,\Diamond_2)+\LK{5}*\LK{5}$ has the finite model property and decidable.
\end{example}
\begin{remark}
Dynamic logics based on atomic modalities satisfying $\LK{5}$ are considered in the context of
epistemic logic and logical investigation of game theory, see, e.g., \cite{FittingGames2011}
(in this context, the axiom $\Diamond p\imp \Box \Diamond p$ is usually addressed as {\em negative introspection}).
\end{remark}

From Theorems \ref{thm:ADFforCPDL} and \ref{thm:mainTransferNew}, we obtain:

\begin{corollary}\label{cor:main}
Let $\Al$ be a finite set,  $L_1, \ldots, L_n$ be logics
such that $L_1 * \ldots * L_n\subseteq \Fm(\Al)$.
If $L_1, \ldots, L_n$ admit definable filtration, then $\CPDL(\At)+L_1 * \ldots * L_n$ has the finite model property.
If also $L_1, \ldots, L_n$ are finitely axiomatizable,
then $\CPDL(\At)+L_1 * \ldots * L_n$ is decidable.
\end{corollary}

\subsection{Examples}\label{sec:LTimpADF}
As we mentioned, for the logics
$\LK{},~\logicts{T},~\LK{4},~\LS{4},~\LS{5}$, as well as for many others, strict filtrations are well-known, see e.g., \cite[Chapter 5]{Ch:Za:ML:1997}.
 In fact, there is a continuum of modal logics that admit strict filtration.
In \cite{bezhanishvili2016stable}, a family of modal logics called {\em stable} was introduced.
Logics $\logicts{T}$ or $\LK{}+\{\Diamond\top\}$ are examples of stable logics.
Every stable logic admits strict filtration, which follows
from \cite[Theorem 7.8]{bezhanishvili2016stable}, and  there are
continuum many stable logics \cite[Theorem 6.7]{bezhanishvili2018stable}.

\begin{remark}
Stable logics were also used to construct decidable extensions of $\PDL$.
Namely, in \cite{IlinAiML2016}, it was announced that
extensions of $\PDL$ with axioms of stable logics have the finite model property.\later{DC with Nick and Ilin}
\end{remark}
\noindent
Another class of logics
 that admit strict filtration
 are logics given by canonical {\em MFP-modal formulas} introduced in \cite[Section 4.2]{KikotShapZolAiml2020}.

\bigskip

There are logics that do not admit strict filtration, but admit definable filtrations.
Consider the family of logics $\LK{} + \{ \Diamond^m p \to \Diamond p\}$ for $m \geq 3$.
These logics are Kripke complete, and their frames are characterized by the  conditions
\def\Rmin{R_{\sim_\Delta}}
\begin{equation}\label{eq:pretrM}
\forall x \, \AA  y \: (x R^m y \Rightarrow x R y);
\end{equation}
moreover, all these logics admit definable filtration \cite[Theorem 8]{Gabbay:1972:JPL}:
for a given $\Gamma$ and a model, the required filtration can be built by letting
$\Delta= \{\Diamond^{i} \vf \mid \vf \in\Gamma\;\&\; i \leq m - 2 \}$.
However, these logics do not admit strict filtration.
We will illustrate it with the case when $m = 3$, one can generalize it for any $m \geq 3$.\hide{
(to define the relation,  let $S$ be the minimal filtered relation
$\Rmin$, and put
put $\ff{R}=\bigcup_{0<k<\omega}S^{1+(m-1)k}$)}
\begin{example}
$L = {\bf K} + \{\Diamond \Diamond \Diamond p \to \Diamond p\}$ does not admit strict filtration.
\end{example}

\begin{proof}
  Consider a five-element model $M = (W, R, \vartheta)$, where the binary relation is defined by the following figure
\medskip

    \xymatrix@R-0.6cm{
  &&&& x \ar[r] & y \\
  &&&&& y' \ar[r] & z \ar[r] & u
  }
\medskip
\noindent
($R$ is assumed to be irreflexive),
and 
$$
\vartheta(p) = \{ x \}, \quad \vartheta(q) = \{ y, y' \}, \quad \vartheta(r) = \{ u \}.
$$
By \eqref{eq:pretrM}, the frame of $M$ validates  $\Diamond \Diamond \Diamond p \to \Diamond p$,
and so $M$ is a model of the logic $L$.
Let $\Gamma = \{p, q , r, \Diamond r \}$. Assume
that $\ff{M} = (W{/}{\sim_\Gamma}, \ff{R}, \ff{\vartheta})$ is a $\Gamma$-filtration of $M$ and show that
$\ff{M}$ is not an $L$-model.
Notice  that $y$ and $y'$ are $\sim_\Gamma$-equivalent, and hence the quotient $W{/}{\sim_\Gamma}$
consists of four elements $[x],[y](=[y']),[z],[u]$.
Since $\ff{R}$ contains the minimal filtered relation, we have $[x]\ff{R}[y]\ff{R}[z] \ff{R} [u]$. For the sake of contradiction, assume that $\ff{M}\mo L$.
We have $\ff{M},[u] \mo r$, and so $\ff{M},[x] \mo \Diamond\Diamond\Diamond r$.
Then $\ff{M},[x]\mo \Diamond r$ by assumption.
Since $\Diamond r\in\Gamma$ and $\ff{M}$ is a $\Gamma$-filtration of $M$,
we have $M,x\mo \Diamond r$, which contradicts the definition of $M$. Hence $\ff{M}$ is not an $L$-model.
\end{proof}

A continuum of logics that admit definable filtration are locally tabular logics.
Recall that a logic $\vL$ is {\em locally tabular},
if, for every finite $k$, $\vL$ contains only a finite number of pairwise nonequivalent
formulas in a given $k$ variables.
Well-known examples of locally tabular modal logics are $\LK{5}$ \cite{nagle_thomason_1985} and so its extensions (e.g., $\LK{45}$, $\LS{5}$),
or the {\em difference logic} $\LK{}+\{p\imp \Box\Diamond p,~\Diamond\Diamond p\imp \Diamond p\vee p\}$ \cite{esakia2001weak}.
\later{DC ref for Esakia;
see
Simple weakly transitive modal algebras
https://go.gale.com/ps/i.do?id=GALE
}


Let $M=(W,(R_a)_{a\in\Al},\theta)$ be a model of a locally tabular logic $L$,
$\Gamma\subset \Fm(\Al)$ a finite $\Sub$-closed set of formulas. 
Let $V$ be the set of all variables occurring in $\Gamma$,
and let $\Delta$ be the set of all $\Al$-formulas with variables in $V$. 
Let $\fkCan{V}$ be the canonical frame of $L$ built from
maximal $\vL$-consistent subsets of $\Delta$; the canonical relations are defined in the standard way.
Consider the maximal $\Delta$-filtration $\ff{M}$ of $M$  with the carrier $W/{\sim_\Delta}$;
in \cite{Shehtman:AiML:2014},
such filtrations are called {\em canonical}.
Since $L$ is locally tabular, $\ff{M}$ is finite. The frame $\ff{F}$ of $\ff{M}$
is isomorphic to a generated subframe
of $\fkCan{V}$, see, e.g.,
\cite{Shehtman:AiML:2014} for details.\later{DC}
Since $L$ is locally tabular, $\fkCan{V}$ is finite, and so
$\fkCan{V}\mo L$. It follows that $\ff{M}\mo L$, as required.\later{more details} Hence, we have
\obsolete{
$\ff{M}$ is a p-morphic image of $M$ (for details, see \cite[Proposition 2.32]{Shehtman:AiML:2014}).
\ISH{This ref is not exact. The argument must be that the frame of M is a p-morphic image of a generated subframe of the k-canonical frame}
}
\begin{theorem}[Corollary from \cite{Shehtman:AiML:2014}]\label{thm:LFimpliesADF}
If $L$ is locally tabular, then $L$ admits definable filtration.
\end{theorem}

\hide{
Note that if a logic is locally tabular, it does not have to admit strict filtration, the example is the logic $~\LK{5}$, which is locally finite, but it does not admit strict filtration though.
}

Putting the above examples together, we obtain the following instance of Corollary \ref{cor:main}.
\begin{corollary}\label{cor:final}
Let $\Al$ be a finite set,  $L_1, \ldots, L_n$ be logics
such that $L_1 * \ldots * L_n\subseteq \Fm(\Al)$.
If each $L_i$
is
\begin{itemize}
\item
one of the logics
$$\LK{},~\logicts{T},~\LK{4},~\LS{4},
~
\LK{} + \{\Diamond^m p \to \Diamond p\}~(m \geq 1),$$ or an extension of any of these logics with a variable-free formula,
\item  locally tabular (e.g., $\LK{5},~\LK{45},~\LS{5}$, the difference logic), or
\item  a stable logic, or
\item axiomatizable by canonical MFP-modal formulas,
\end{itemize}
then $\CPDL(\At)+L_1 * \ldots * L_n$ has the finite model property.
If also all $L_i$ are finitely axiomatizable, then $\CPDL(\At)+L_1 * \ldots * L_n$ is decidable.
\end{corollary}

\obsolete{
\begin{corollary}
Let $n$ and $\At$ be finite, $L_1, \ldots, L_n, L$ be modal logics
such that $L_1 * \ldots * L_n *L \subseteq \Fm(\At)$.
\begin{enumerate}
\item
If $L_1, \ldots, L_n, L$ admit definable filtration, then $\CPDL(\At)+L_1 * \ldots * L_n*L$ has the finite model property.
\item If each $L_i$
is one of the logics
$\LK{},~\logicts{T},~\LK{4},~\LK{5},~\LK{45},~\LS{4},~\LS{5},~\LK{}+\{p\imp \Box\Diamond p\},~\LK{}+\{\Di\top\},
\LK{4}+\{\Di\top\}$,
$\LK{} + \{\Diamond^m p \to \Diamond p\}$ for $m \geq 1$, or $L_i$or any stable logic
and $L$ admits definable filtration,
then  $\CPDL(\At)+L_1 * \ldots * L_n*L$ has the finite model property.
If also $L$ and each $L_i$ are finitely axiomatizable, then $\CPDL(\At)+L_1 * \ldots * L_n*L$ is decidable.
\end{enumerate}
\end{corollary}
}

\section{Acknowledgement}

The authors wishes to thank Nick Bezhanishvili for valuable discussions.
The authors are also grateful to the referees for their comments on an earlier version of this paper.

\bibliographystyle{alpha}
\bibliography{filtration}
\end{document}